\documentclass[pdftex,a4paper,oneside]{amsart} 

\usepackage[utf8]{inputenc}
\usepackage[english]{babel}
\usepackage[T1]{fontenc}
\usepackage{csquotes}
\usepackage[noabbrev, capitalise]{cleveref}

\usepackage{amsmath}
\usepackage{amstext}
\usepackage{amssymb}
\usepackage{amsfonts}
\usepackage{mathrsfs}
\usepackage{amssymb}
\usepackage{amsthm}
\allowdisplaybreaks
\usepackage{etoolbox}
\usepackage{needspace} 

\usepackage{geometry}
\usepackage[labelformat=simple,belowskip=-.5em]{subcaption}

\usepackage{booktabs}
\usepackage{enumerate}
\usepackage[shortlabels]{enumitem}
\usepackage{calc}
\newlength{\IdentationTDConditions}
\newlength{\IdentationTDExtraCondition}
\usepackage{color}
\definecolor{gray}{rgb}{.4,.4,.4}

\usepackage{graphics}

\newcommand{\diam}{\operatorname{diam}}
\newcommand{\bigO}{\operatorname{\mathcal{O}}}

\newcommand{\X}{\mathcal{X}}
\newcommand{\Bceil}[1]{\left\lceil#1\right\rceil}
\newcommand{\Bfloor}[1]{\left\lfloor#1\right\rfloor}
\newcommand{\tw}{\operatorname{tw}}

\newcommand{\ksec}[1]{\mbox{\(#1\)-section}}
\newcommand{\Ksec}[1]{\mbox{\(#1\)-Section}}
\newcommand{\minsec}[2]{\operatorname{MinSec}_{#1}(#2)}

\newcommand{\cut}[1]{\mbox{\(#1\)-cut}}

\newtheorem{thm}{Theorem}
\newtheorem{lemma}[thm]{Lemma} 
\newtheorem{cor}[thm]{Corollary} 
\newtheorem{prop}[thm]{Proposition}
\newtheorem{defi}[thm]{Definition}

\makeatletter
\providerobustcmd*{\bigcupdot}{%
  \mathop{%
    \mathpalette\bigop@dot\bigcup
  }%
}
\newrobustcmd*{\bigop@dot}[2]{%
  \setbox0=\hbox{\(\m@th#1#2\)}%
  \vbox{%
    \lineskiplimit=\maxdimen
    \lineskip=-0.7\dimexpr\ht0+\dp0\relax
    \ialign{%
      \hfil##\hfil\cr
      \(\m@th\cdot\)\cr
      \box0\cr
    }%
  }%
}
\makeatother

\begin{document}


\title[Minimum \(k\)-Section in Trees]{Approximating the Minimum $k$-Section Width\\ in Bounded-Degree Trees with Linear Diameter} 

\author{Cristina G.\ Fernandes}
\address[Cristina~G.~Fernandes]{Instituto de Matem\'atica e Estat\'{\i}stica \\
  Universidade de S\~ao Paulo\\
  Rua do Mat\~ao~1010, 05508--090\\
  S\~ao Paulo\\
  Brazil}
\email{cris@ime.usp.br}
\thanks{The first author was partially supported by CNPq Proc.~308523/2012-1 and 477203/2012-4, FAPESP 2013/\mbox{03447-6}, and Project MaCLinC of NUMEC/USP}
 
\author{Tina Janne Schmidt}
\author{Anusch Taraz}
\address[Tina~Janne~Schmidt and Anusch~Taraz]{Institut f\"ur Mathematik\\
  TU Hamburg\\
  Am Schwarzenberg-Campus~3\\
  21073 Hamburg\\
  Germany}
\email{tina.janne.schmidt@tuhh.de \textnormal{and} taraz@tuhh.de}
\thanks{The second author gratefully acknowledges the support by the Evangelische Studienwerk Villigst e.V.\\ 
The research of the three authors was supported by a PROBRAL CAPES/DAAD Proc.~430/15 (February 2015 to December 2016, DAAD Projekt-ID 57143515). \\
Some of the results that are proven in this work have already been announced in extended abstracts at LAGOS~2015 and EuroComb~2015 \cite{EuroComb2015, LagosKSec}}


\date{\today}




\maketitle          
\thispagestyle{empty}


\begin{quote}
\footnotesize
\textsc{Abstract.} Minimum \mbox{\(k\)-Section} denotes the NP-hard problem to partition the vertex set of a graph into~\(k\) sets of sizes as equal as possible  while minimizing the cut width, which is the number of edges between these sets. When~\(k\) is an input parameter and~\(n\) denotes the number of vertices, it is NP-hard to approximate the width of a minimum \mbox{\(k\)-section} within a factor of~\(n^c\) for any~\(c <1\), even when restricted to trees with constant diameter. Here, we show that every tree~\(T\) allows a \mbox{\(k\)-section} of width at most~\( (k-1) (2 + 16n / \operatorname{diam}(T) ) \Delta(T) \). This implies a polynomial-time constant-factor approximation for the Minimum \Ksec{k} Problem when restricted to trees with linear 
diameter and constant maximum degree. Moreover, 
we extend our results from trees to arbitrary graphs with a given tree decomposition.
\end{quote}

\section{Introduction}
\label{secIntroduction}

\subsection{The Minimum \mbox{$k$-Section} Problem}
\label{subsecMinKSec} 

We start with a few definitions. A \emph{\ksec{k}} in a graph~\(G\) is a partition~\((V_1, V_2, \ldots, V_{k})\) of its vertex set into~\(k\) sets whose sizes are as close to equal as possible. 
The \emph{width} of a \ksec{k}~\((V_1, \ldots, V_k)\) is the number of edges between the sets~\(V_{\ell}\) and is denoted by~\(e_G(V_1, \ldots, V_k)\). The minimum width among all \ksec{k}s in a graph~\(G\) is denoted by~\(\minsec{k}{G}\). The goal of the \emph{Minimum \Ksec{k} Problem} is to compute~\(\minsec{k}{G}\) and a corresponding \ksec{k} for a given graph~\(G\) and an integer~\(k \geq 2\). This problem has many applications, e.g.~in parallel computing, when tasks have to be evenly distributed to processors while minimizing the communication cost.  

The special case where~\(k=2\), which is also called the \emph{Minimum Bisection Problem}, is known to be \mbox{NP-hard} for general graphs, see~\cite{GareyJohnsonStockmeyer}. Jansen et al.~\cite{JansenKarpinski} use dynamic programming to compute a minimum bisection in a tree on~\(n\) vertices in~\(\mathcal{O}( n^3)\) time. Their method can be turned into an algorithm for computing a minimum \mbox{\(k\)-section} in trees, whose running time is polynomial in~\(n\) but not in~\(k\). However, when~\(k\) is part of the input,~\(\minsec{k}{G}\) cannot be approximated in polynomial time within any finite factor on general graphs~\(G\) unless P\(=\)NP, see~\cite{AndreevRaecke}. The reduction presented there is from the strongly \mbox{NP-hard} \mbox{3-Partition} Problem and it is easy to 
adjust it in order to show that the Minimum \Ksec{k} Problem restricted to forests cannot be approximated within any finite factor unless P\(=\)NP.  

Feldmann and Foschini~\cite{FeldmannFoschini} studied minimum \ksec{k}s in trees, and pointed out some counterintuitive behavior even on this rather restricted class of graphs. Moreover, they showed in~\cite{FeldmannFoschini} that the Minimum \mbox{\(k\)-Section} Problem remains APX-hard when restricted to trees with maximum degree at most 7, and that, for any~\(c < 1\), it is \mbox{NP-hard} to approximate~\(\operatorname{MinSec}(k,T)\) within a factor of~\(n^c\) for trees with constant diameter. 

\subsection{Results}
\label{subsecResults}

Here, we study the Minimum \Ksec{k} Problem in trees and focus on bounded-degree trees with linear diameter. Moreover, we extend our results to tree-like graphs. Our first result gives an upper bound on the width of a minimum \mbox{\(k\)-section} in trees and a corresponding algorithm.

\begin{thm}
  \label{thmTreeKSec}
  For every integer~\(k \geq 2\) and for every tree~\(T\) on~\(n\) vertices, a \ksec{k}~\((V_1, \ldots, V_k)\) in~\(T\) with
  \[ e_T (V_1, \ldots, V_k) \ \leq \ (k-1)  \cdot \left( 2 + \frac{16n}{\diam(T)} \right) \cdot \Delta(T) \]
  can be computed in~\(\bigO (kn) \) time.  
\end{thm}

Here, as usual,~\(\Delta(T)\) and~\(\diam(T)\) denote the maximum degree of~\(T\) and the diameter of \(T\), respectively, where the latter is defined as the length of a longest path in~\(T\). 
Obviously, for \(k \geq n\), any graph on~\(n\) vertices has essentially only one \mbox{\(k\)-section}, so we can assume without loss of generality that \(k<n\) and, hence, the running time in Theorem~\ref{thmTreeKSec} is always polynomial in the input length.  

Let~\(\Delta_0 \in \mathbb{N}\) and~\(d>0\) be two constants. Then, for any tree~\(T\) on~\(n\) vertices with~\(\Delta(T) \leq \Delta_0\) and~\(\diam(T) \geq dn\), the factor \( ( 2 + 16n / \operatorname{diam}(T) ) \Delta(T)\) from the previous theorem is bounded by a constant that depends only on~\(\Delta_0\) and~\(d\). As every \ksec{k} of a tree has width at least~\(k-1\), this yields a constant-factor approximation for~\(\minsec{k}{T}\) for such a class of trees.

\begin{cor}
  \label{corTreeKSec}
  For all~\(\Delta_0 \in \mathbb{N}\) and~\(d>0\), there is a constant~\(c >1 \) such that the following holds. Let~\(\mathcal{T}\) be a class of trees such that every tree~\(T \in \mathcal{T}\) on~\(n\) vertices satisfies~\(\Delta(T) \leq \Delta_0\) and~\(\diam(T) \geq dn\). Then there is a \mbox{\(c\)-approximation} for the Minimum \mbox{\(k\)-Section} Problem restricted to the class~\(\mathcal{T}\). In particular, one can choose~\(c = \left(2+ \frac{16}{d}\right) \Delta_0\).
\end{cor}

Next, we extend our focus in this paper in two ways. On the one hand, we can improve the upper bound so that it becomes polylogarithmic in~$n/\diam(T)$. Moreover, we now move from trees to graphs with a given tree decomposition. 
Here, instead of bounding the width of a \ksec{k} in terms of the diameter, we define a parameter~\(r(T, \mathcal{X} )\) that roughly measures how close the tree decomposition~\( (T, \mathcal{X} )\) is to a \emph{path decomposition}, which is defined as a tree decomposition where the decomposition tree is a path. For example, suppose that we were given a graph~\(G\) with a path decomposition~\( (P, \mathcal{X} )\) of~\(G\) of width~\(t-1\). Then it is easy to see that~\(G\) allows a bisection of width at most~\(t \Delta(G)\) by walking along the path~\(P\) until we have seen~\(\frac{1}{2}n\) vertices of~\(G\) in the clusters, and then bisecting~\(G\) there. 
In other words, finding a \ksec{k} of the graph~\(G\) of small width becomes easier when there is a path in the tree decomposition of~\(G\) whose clusters contain many of the vertices of~\(G\).

For the precise definition of the parameter~\(r(T, \mathcal{X} )\), 
consider a tree decomposition~\( (T, \mathcal{X} )\) of a graph~\(G=(V,E)\) with~\({ \mathcal{X}  = (X^i)_{i \in V(T)}}\). The \emph{relative weight of a heaviest path} in~\((T,\X)\) is denoted by
\begin{equation*}
  r(T,\X) := \; \frac{1}{n} \;  \max_{P\subseteq T \text{ path}} \left| \textstyle\bigcup_{i \in V(P)} X^i\right|,
\end{equation*}
where~\(n\) denotes the number of vertices of~\(G\). Observe that~\(\frac{1}{n} \leq r(T,\X) \leq 1\). Moreover, define the \emph{size} of~\((T,\X)\) as~\( \| (T, \mathcal{X} ) \| := |V(T)| + \sum_{i\in V(T)} |X^i| \), which roughly measures the encoding length of~\((T,\X)\).

\begin{thm}
  \label{thmGenKSec}
  For every integer~\(k \geq 2\), for every graph~\(G\) and every tree decomposition~\( (T,\X) \) of~\(G\) of width at \mbox{most~\(t-1\)},  a \ksec{k}~\( (V_1, \ldots, V_k) \) in~\(G\) with
  \[ e_G(V_1, \ldots, V_k) \ \leq \ \tfrac{1}{2} (k-1) t \Delta(G) \left( \left( \log_2 \left( \frac{1}{r(T,\X)} \right) \right)^2 + 11 \log_2 \left( \frac{1}{r(T,\X)} \right) + 24 \right)\]
  can be computed in~\(\bigO(k \| (T,\X)\| )\) time, when the tree decomposition~\((T,\X)\) is provided as input. 
\end{thm}

Again, as in the case of trees in Theorem~\ref{thmTreeKSec} and Corollary~\ref{corTreeKSec}, a constant-factor approximation for a certain class of tree-like graphs is obtained. More precisely, fix~\(\Delta_0 \in \mathbb{N}\),~\(0 < r \leq 1\),~\(t_0 \in \mathbb{N}\) and define~\(c := \frac{1}{2} t \left( \left(\log_2 \left( \frac{1}{r} \right)\right)^2 + 11 \log_2 \left( \frac{1}{r}\right) + 24 \right) \Delta_0  \). Consider a class~\(\mathcal{G}\) of connected graphs such that every graph~\(G \in \mathcal{G}\) satisfies~\(\Delta(G) \leq \Delta_0\) and allows a tree decomposition~\((T,\X)\) of width at most~\(t_0-1\) and~\(r(T,\X) \geq r\). Then, there is a \mbox{\(c\)-approximation} for the Minimum \Ksec{k} Problem restricted to the graph class~\(\mathcal{G}\).


\subsection{Related Work}
\label{subsecRelWork}

The results presented here rely on our earlier work concerning the Minimum Bisection Problem, in particular the following theorem. 

\begin{thm}[Theorem~1 in \cite{PaperGenTreeDiam}]
  \label{thmTreeBis}
  For every tree~\(T\) a bisection~\((B, W)\) in~\(T\) with
  \begin{align*}
    e_T(B, W) \ \leq \ \frac{8 n}{\diam(T)} \Delta(T)
  \end{align*}
  can be computed in~\(\bigO(n)\) time. 
\end{thm}

Although Theorem~\ref{thmTreeKSec} and Theorem~\ref{thmTreeBis} look quite similar, it does not seem possible to directly apply Theorem~\ref{thmTreeBis} to yield a recursive construction of a \mbox{\(k\)-section} that satisfies the bound presented in Theorem~\ref{thmTreeKSec}.  Indeed, it is known that, even when~\(k\) is a power of~\(2\), the natural approach to construct a \mbox{\(k\)-section} in a graph  by recursively constructing bisections can give solutions far from the optimum, even when a minimum bisection is used in each step, see~\cite{SimonTeng}. Furthermore, in the setting that we are considering here, i.e.,~bounded-degree trees with linear diameter, nothing is known about the diameter in the two subgraphs that are produced by Theorem~\ref{thmTreeBis}. So, after the first iteration, the diameter of one part of the bisection could be as low as~\( \mathcal{O}( \log n )\), and indeed such parts can be produced by the algorithm contained in Theorem~\ref{thmTreeBis}. For example, consider the tree~\(T\) obtained from a perfect ternary tree~\(T'\) on~\(\frac{1}{2}n\) vertices and a path~\(P'\) on~\(\frac{1}{2}n\) vertices by inserting an edge joining the root~\(r\) of~\(T'\) and a leaf of~\(P'\). The algorithm%
\footnote{To give some more details, this algorithm computes a longest path~\(P\) in~\(T\), which must contain all vertices from~\(P'\) as well as~\(r\). Then, it computes a \mbox{\(P\)-labeling}, as introduced in~\cref{subsecTreeLemma} ahead, which can label the vertices of~\(T'\) with~\(1, \ldots, \frac{1}{2}n\) and the vertices in~\(P'\) with~\(\frac{1}{2} n+1, \ldots, n\). Then, the algorithm checks whether there is a vertex~\(v \in V(P)\) such that~\(v+\frac{1}{2}n \in V(P)\) where we identified the vertices with their labels. This is the case for~\(v=r\) and, hence, the bisection~\((B,W)\) with~\(B= V(P')\) and~\(W=V(T')\) is output.  For further details see Section~6.1 in \cite{thesisTina}.} contained in Theorem~\ref{thmTreeBis} 
can output the bisection~\((B,W)\) with~\(B= V(P')\) and~\(W=V(T')\), which is in fact the unique minimum bisection in~\(T\). In the next round, a bisection in~\(T'\) is needed, which has width~\(\Omega( \log n)\). Thus, a \ksec{4} of width~\(\Omega(\log n)\) is obtained, whereas \cref{thmTreeKSec} promises that~\(T\) allows a \ksec{4} of constant width. 

Observe that, when only the universal bound~\(\Omega(\log n)\) is available for the diameter of a bounded-degree tree on~\(n\) vertices, then \cref{thmTreeBis} yields a bound of~\(\bigO \left( \frac{n}{\log n} \right)\) for the width of a bisection.

\subsection{Further Remarks}

The results presented in~\cref{thmTreeKSec} and~\cref{thmGenKSec} do not only hold for \ksec{k}s but also for cuts~\((V_1, V_2, \ldots, V_k)\), i.e., partitions of the vertex set of the considered graph, where the sizes of the sets are specified as input. Furthermore, the bound in \cref{thmTreeKSec} can be improved to 
\begin{equation}
  \label{eqTreeKSecImpr}
   e_T (V_1, V_2, \ldots, V_k) \ \leq \  \tfrac{1}{2}(k-1)  \left( \left( \log_2 \left( \frac{n}{\diam(T)}  \right) \right)^2 + 9 \log_2 \left( \frac{n}{\diam(T)} \right) +18  \right) \Delta(T).
\end{equation}
Observe that this is slightly stronger than the bound implied by \cref{thmGenKSec} and the fact that every tree~\(T'\) allows a tree decomposition~\((T,\X)\) of width at most one with~\(r(T,\X) \geq \frac{1}{n} (\diam(T)+1) \) and~\(\|(T,\X)\| = \bigO(n)\). Moreover, extensions to \ksec{k}s in trees with weighted vertices have recently been investigated, see~\cite{BScThesisFabian}.

\subsection{Organization of the Paper}

\cref{secPrelim} introduces some notation for cuts in general graphs as well as some tools for trees, which will then be used in \cref{secProofTree} to show \cref{thmTreeKSec}. Moreover, in \cref{subsecTreeRemarks}, it is argued that the bound on the width of the \ksec{k} in \cref{thmTreeKSec} can be improved as claimed in~\eqref{eqTreeKSecImpr}.  \cref{secProofGen} concerns tree-like graphs. First, in \cref{subsecPrelTreeDec}, the notation for tree decompositions is settled and selected tools for tree-like graphs are presented. Second, the proof for \cref{thmGenKSec} is given in \cref{subsecGenProofThm}-\ref{subsecGenAlgo}. 
Since the proofs in \cref{secProofTree} and \cref{secProofGen} follow the same ideas, we do not repeat the full details for the case of tree-like graphs in \cref{secProofGen} but focus on the aspects that become more involved when dealing with a tree decomposition and refer to \cref{secProofTree} whenever possible.

\section{Preliminaries}
\label{secPrelim}

First, for~\(k \in \mathbb{N}\), define~\([k]:= \{1,2, \ldots, k\}\).  Moreover, for a real~\(c\) denote by~\(\Bfloor{c}\) the largest integer~\(\leq c\) and denote by~\(\Bceil{c}\) the smallest integer~\(\geq c\).  Consider an arbitrary graph~\(G=(V,E)\) on~\(n\) vertices. For a set~\( \emptyset \neq S \subseteq V\), we use~\(G[S]\) to denote the subgraph of~\(G\) induced by~\(S\) and for~\(S \subsetneq V\) we define~\(G-S := G[V\setminus S]\). A \emph{cut} in~\(G\) is a partition~\((V_1, V_2, \ldots, V_{k})\) of~\(V\), where~\(k \in \mathbb{N}\) is arbitrary and empty sets are allowed. An edge~\(e=\{v,w\}\) of~\(G\) is \emph{cut} by a cut~\((V_1, \ldots, V_k)\) if there are distinct~\(\ell, h \in [k]\) with~\(v \in V_{\ell}\) and~\(w \in V_h\). The \emph{width} of a cut~\((V_1, \ldots, V_k)\) in~\(G\) is defined as the number of edges of~\(G\) that are cut by~\((V_1, \ldots, V_k)\) and is denoted as~\(e_G(V_1, \ldots, V_k)\). For~\(k \in \mathbb{N}\), a \emph{\ksec{k}}~\((V_1, V_2, \ldots, V_k)\) in~\(G\) is a cut~\((V_1, \ldots, V_k)\) in~\(G\) with~\(\Bfloor{\frac{n}{k}} \leq |V_{\ell}| \leq \Bceil{\frac{n}{k}}\) for all~\(\ell \in [k]\).  A \ksec{k}~\((V_1, \ldots, V_k)\) in~\(G\) that satisfies~\(e_G(V_1, \ldots, V_k) \leq e_G(V_1', \ldots, V_k')\) for all \ksec{k}s~\((V_1', \ldots, V_k')\) in~\(G\) is called a \emph{minimum \ksec{k}} in~\(G\) and its width is denoted by~\(\minsec{k}{G}\).

Recall that the \emph{diameter} of a tree~\(T\) is the length of a longest path~\(P\) in~\(T\), i.e.,~\(\diam(T) =|E(P)|\). In the following, we need to compare the diameter of trees with different numbers of vertices and during our construction also non-connected forests may arise.  Consider a forest~\(G\) on~\(n\) vertices and denote by~\(G_1, \ldots, G_{\ell}\) the connected components of~\(G\). Then, the \emph{relative diameter} of~\(G\) is defined as
\begin{equation*}
  \diam^*(G) \ := \ \frac{1}{n} \sum_{h \in [\ell]} (\diam(G_h)+1).
\end{equation*}
The term~\(\diam(G_h)+1\) denotes the number of vertices on a longest path in the component~\(G_h\) and for a tree~\(T\) the relative diameter equals the proportion of vertices on a longest path in~\(T\). Using this notation, we can state the version of~\cref{thmTreeBis} which we will employ in \cref{secProofTree}. It follows from Theorem~1 in~\cite{PaperGenTreeDiam} and the comments in Section~1.4 there. When considering a cut in a graph~\(G\) with exactly two sets, we use~\(B\) and~\(W\) for these sets and refer to them as the black and the white set of the cut. 

\begin{thm}[similar to Theorem~1 in \cite{PaperGenTreeDiam}]
  \label{thmTreeDiam}
  For every forest~\(G\) on~\(n\) vertices, for every~\(m \in [n]\), a cut~\((B,W)\) with~\(|B| =m \) and
  \[e_G(B,W) \leq \frac{8 }{\diam^*(G)} \Delta(G)\]
  can be computed in~\(\bigO(n)\) time. 
\end{thm}

The previous theorem allows to cut off an arbitrary number of vertices in a bounded-degree tree and guarantees a small cut width if the diameter is large. The next tool relaxes the size-constraint on the set~\(B\). Let~\(G =(V,E)\) be a graph. For an integer~\(m\), a cut~\((B,W)\) in~\(G\) is called an \emph{approximate \cut{m}} if \( \frac{1}{2}m \leq |B| \leq m\). The next lemma states that every bounded-degree tree allows an approximate cut of small width, even when the diameter is small. 

\begin{lemma}[approximate cut in forests]
  \label{lemmaApproxCutTree}
  Let~\(T\) be a tree on~\(n\) vertices and fix a vertex~\(v \in V(T)\). For every integer~\(m \in [2n-2]\), an approximate \cut{m}~\((B,W)\) with~\(e_T(B,W) \leq \Delta(G)\) and~\(v \in W\) can be computed in~\(\bigO(n)\) time. 
\end{lemma}

The previous lemma is similar to Lemma~7 in~\cite{PaperGenTreeDiam} but as the bound on~\(e_T(B,W)\) claimed here is smaller we present a sketch of its proof.

\begin{proof}[Sketch of Proof for \cref{lemmaApproxCutTree}.]
  Let~\(T=(V,E)\), \(n\), \(m\), and~\(v\) be as in the statement. If~\(m \geq n-1\), define~\(B := V\setminus\{v\}\). Otherwise, root~\(T\) in~\(v\) and, for~\(w \in V\), denote by~\(D_w\) the set of descendants of~\(w\). It is easy to check that there is a vertex~\(x\) such that~\(|D_x| > m\) and~\(|D_y| \leq m\) for all children~\(y\) of~\(x\). If there is a child~\(y\) of~\(x\) with~\(|D_y| \geq \frac{1}{2} m\), define~\(B:= D_y\). Otherwise, the set~\(B\) can be constructed greedily from the sets~\(D_y\) where~\(y\) is a child of~\(x\). Then, the cut~\((B,W)\) with~\(W:= V\setminus B\) has the desired properties and can easily be computed in linear time. 
\end{proof}

\section{Minimum \(k\)-Section in Trees}
\label{secProofTree}

The aim of this section is to prove \cref{thmTreeKSec} about \ksec{k}s in trees. \cref{subsecTreeThm} introduces the main lemma that immediately implies the existence part of the desired theorem and \cref{subsecTreeLemma} presents the proof of the main lemma. All algorithmic aspects of \cref{thmTreeKSec} are presented in \cref{subsecTreeAlgo}. Finally, \cref{subsecTreeRemarks} argues that the bound on the width of the \ksec{k} in \cref{thmTreeKSec} can be improved as stated in~\eqref{eqTreeKSecImpr}.

\subsection{Proof of \cref{thmTreeKSec}}
\label{subsecTreeThm}

The aim of this section is to prove our main result for trees. As mentioned in \cref{subsecRelWork}, constructing a \ksec{k} by bisecting the graph repeatedly does not yield the bound provided by \cref{thmTreeKSec}. So we follow a different approach: The main idea is to cut off one set for the \ksec{k} at a time while ensuring that the relative diameter of the remaining forest does not decrease. This is made precise by the next lemma, which looks similar to \cref{thmTreeDiam} but is more powerful, as it contains additional information on the relative diameter of the subgraph induced by the white set of the cut. 

\begin{lemma}
  \label{lemmaTreeCutPresDiam}
  For every forest~\(G\) on~\(n\) vertices and for every~\(m \in [n-1]\), there is a cut~\( (B, W)\) in~\(G\) with~\(|B| =m\), that satisfies~\( \operatorname{diam}^* ( G[W] ) \geq \operatorname{diam}^*(G)\) and 
  \[ e_G (B, W) \leq \left(2+ \frac{16}{\operatorname{diam}^*(G)} \right) \Delta(G).\]
\end{lemma}

Observing that~\(\diam^*(T) \geq \frac{1}{n} \diam(T)\) now yields the existence part of \cref{thmTreeKSec}.

\subsection{Proof of \cref{lemmaTreeCutPresDiam}}
\label{subsecTreeLemma}

Consider a forest~\(G\) on~\(n\) vertices and fix an integer~\(m \in [n-1]\). The main idea is to apply \cref{thmTreeDiam} to a carefully chosen subgraph~\(\tilde{G} \subseteq G\), which then yields the set~\(B \subseteq V(\tilde{G})\) for the desired cut~\((B,W)\) in~\(G\). On the one hand, \(\tilde{G}\) needs to have large relative diameter such that the bound on the cut width provided by \cref{thmTreeDiam} is low when applied to~\(\tilde{G}\). On the other hand, the relative diameter of the graph induced by the white set of the computed cut will roughly be the relative diameter of~\(G-V(\tilde{G})\), so \(G-V(\tilde{G})\) needs to have a large relative diameter. Note that these two conditions compete against each other.

If~\(\Delta(G) \leq 2\), then~\(\diam^*(G) =1\) and a cut with the desired properties is easy to construct. So assume that~\(\Delta(G) \geq 3\). Without loss of generality, we may assume that~\(G\) is connected as otherwise edges can be added to~\(G\) to obtain a tree whose relative diameter
is equal to~\(\diam^*(G)\) and maximum degree~\(\Delta(G)\). Set~\(d := \diam^*(G)\),  let~\(P = (V_P, E_P)\) be a longest path in~\(G\) and observe that~\(|V_P| = dn\). Denote by~\(x_0\) and~\(y_0\) the two leaves of~\(P\).  When removing all edges in~\(E_P\) from~\(G\), then~\(G\) decomposes into trees, one tree~\(T_v\) for every~\(v \in V_P\). For each~\(v \in V_P\), let~\(T_v' := V(T_v) \setminus\{v\}\). For~\(x\in V(G)\), the unique vertex~\(v \in V_P\) with~\(x \in V(T_v)\) is called the \emph{path-vertex} of~\(x\). When labeling the vertices of a graph with~\( \{1,2, \ldots, n\}\), we say that the vertices in~\(V'\subseteq V(G)\) receive \emph{consecutive labels} if there are~\(\ell, \ell' \in [n]\) such that each vertex in~\(V'\) receives one label in~\( \{\ell, \ell+1, \ldots, \ell' \}\) and every label in~\(\{\ell, \ldots, \ell'\}\) belongs to exactly one vertex in~\(V'\).  A \emph{\mbox{\(P\)-labeling}} of~\(G\) is a labeling of the vertices of~\(G\) with~\(1,2,\ldots, n\) such that the following holds:
\begin{itemize}
 \item For each~\(v \in V_P\), the vertices of~\(T_v\) receive consecutive labels and~\(v\) has the largest label among all vertices in~\(T_v\). 
 \item For all \(v, w \in V_P\) with~\(v \neq w\), if~\(x_0\) is closer to~\(v\) than to~\(w\), then the label of~\(v\) is smaller than the label of~\(w\). 
\end{itemize}
Identify each vertex with its label and consider any number that differs by a multiple of~\(n\) from a label in~\([n]\) to be the same as this label.  When talking about labels and vertices, in particular when comparing them, we always refer to the integer in~\([n]\). For three vertices \(a, b, c \in V\) with~\(a \neq c\), we say that~\(b\) \emph{is between~\(a\) and~\(c\)} if~\(b=a\), \(b=c\), or if starting at~\(a\) and going along the numeration given by the labeling reaches~\(b\) before~\(c\). If~\(a=c\), then we say that~\(b\) is between~\(a\) and~\(c\) if \(b=a=c\). For example, when~\(n=10\), we say that~\(5\) is between~\(1\) and~\(7\), and~\(9\) is between~\(8\) and~\(3\). For technical reasons, we will refer to the pair~\( \{y_0, x_0\}\) as an edge of~\(G\), even though~\(G\) does not contain such an edge. For a vertex~\(v \in V_P\), the vertex~\(w \in V_P\) is the \emph{vertex after~\(v\) on~\(P\)} if the tree~\(T_w\) contains the vertex~\(v+1\). Then, \( \{v,w\}\) is called the \emph{edge after~\(v\) on~\(P\)}. Similarly, in this case,~\( \{v,w\}\) is called the \emph{edge before~\(w\) on~\(P\)} and~\(v\) is called the \emph{vertex before~\(w\) on~\(P\)}. 

For two vertices~\(x,y \in V\), the \emph{\(P\)-distance} of~\(x\) and~\(y\) is defined as
  \[ d_P(x,y) = \left| \left\{ v \in V_P \colon \, \vphantom{'} \text{\(v\) is between~\(x\) and~\(y\), \(v\neq y\)} \right\} \right|.\]
It is easy to see that 
\begin{equation}
  \label{proofLemmaCutPresDiamContin}
  \left| d_P(x,y) - d_P(x+1, y+1) \vphantom{'}\right| \ \leq \ 1 \quad\quad\quad \text{ for all~\(x,y \in V\)}.
\end{equation}
Recall that~\(x_0\) was defined to be an end of~\(P\). Now, define~\(x_{\ell} = x_0 + \ell m\) for all~\(\ell \in  [n]\). Then, \(x_n = x_0 + nm = x_0\) and
  \[ \sum_{\ell=1} ^n d_P(x_{\ell-1}, x_{\ell}) \ = \ m |V_P| \ = \ m dn.\] 
Thus, there are two vertices~\(x', x'' \in V\) with 
  \[d_P (x', x' + m) \leq dm \quad \quad \text{and} \quad \quad d_P(x'', x''+m) \geq dm.\] 
The fact that~\(d_P(x,y)\) is an integer for all~\(x\), \(y \in V\) and~\eqref{proofLemmaCutPresDiamContin} implies that there is a vertex~\(x^* \in V\) with~\(d_P (x^*,x^* + m) = \Bfloor{dm}\). Let~\(p_{x^*}\) and~\(p_{x^*+m}\) be the path-vertices of~\(x^*\) and~\(x^*+m\), respectively. Define~\(h := \min\{ p_{x^*} - x^*, p_{x^* +m}- (x^* +m)\}\). Set~\(v := x^* +h\) and note that~\(v \in V_P\) or~\(v+m \in V_P\). Furthermore, as~\(v+m\) is not counted in~\(d_P(v,v+m)\), we have \(d_P (v,v + m) = d_P(x^*, x^* +m ) = \Bfloor{dm}\). Define
  \[ M := \{u \in V\colon\, u \text{ is between~\(v\) and~\(v+m-1\)}\}. \]
  
In the following figures, the path~\(P\) will be drawn in the top and the trees~\(T_u\) for~\(u \in V_P\) are drawn underneath~\(P\). The vertices in~\(P\) that are counted in~\(d_P(v,v+m)\) will be colored gray.

\begin{figure}
  \begin{subfigure}{0.52 \textwidth}
  \begin{center}
    \includegraphics{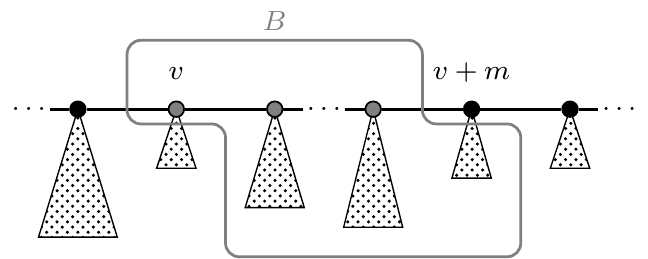}
  \end{center}
  \caption{Case 1, where~\(v \in V_P\) and~\(v+m \in V_P\).}
  \label{figTreeCutPresDiamCase1}
  \end{subfigure}
  \begin{subfigure}{0.45 \textwidth}
  \begin{center}
    \includegraphics{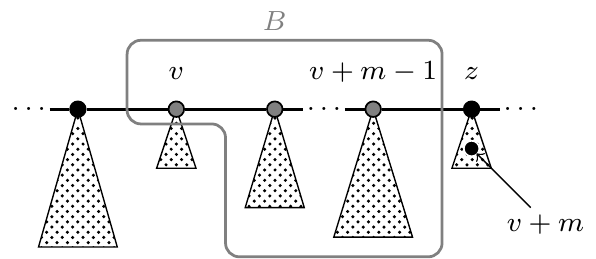}
  \end{center}
  \caption{Case 2a, where~\(v \in V_P\) and~\(v+m-1 \in V_P\).}
  \label{figTreeCutPresDiamCase2a}
  \end{subfigure}
    
  \caption{Construction of the black set in Case~1 and Case~2a.}
  \label{figTreeCutPresDiamCase1and2a}
\end{figure}

\textbf{Case 1:} \(v \in V_P\) and \(v+m \in V_P\). \\
Define~\( B := M \) and~\(W := V \setminus B\). Then~\(e_G (B,W) \leq 2\Delta(G)\), see \cref{figTreeCutPresDiamCase1}. Furthermore, \(|B| =m \)  and
  \[ \diam^*( G[W]) \ \geq \ \frac{ | V_P \cap W| } {|W|} \ = \ \frac{|V_P| - d_P(v, v+m)}{|W|} \ \geq \ \frac{ dn - dm }{n-m} \ = \ d.\]

\textbf{Case 2:} \(v \in V_P\) and \(v+m \not\in V_P\). \\
Let~\(z\) be the path-vertex of~\(v+m\). Observe that~\(z \not\in M\) as otherwise~\(V_P \subseteq M\) and~\(\Bfloor{dm} = |V_P| =dn\), which contradicts~\(m \in [n-1]\). None of the edges in~\(T_z\) is cut by~\( (M, V \setminus M)\) when~\(v+m-1 \in V_P\), i.e., when~\(v+m-1\) is the vertex before~\(z\) on~\(P\), and this case is treated separately for technical reasons. 

\textbf{Case 2a:}  \(v+m-1 \in V_P\). \\
Analogously to Case~1, the cut~\( (B,W)\) with~\(B := M\) and~\(W := V \setminus M\) satisfies all requirements, see \cref{figTreeCutPresDiamCase2a}.

\begin{figure}
  \begin{center}
    \includegraphics{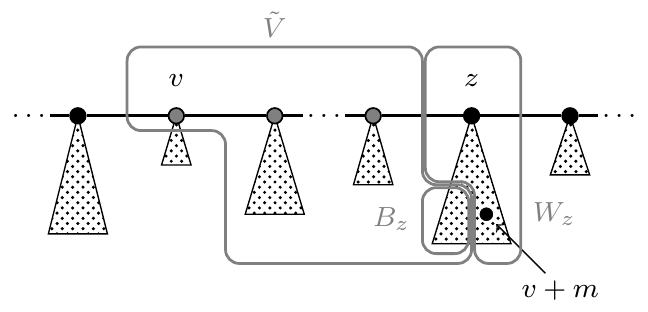}
  \end{center}
  \caption{Construction of~\(\tilde{V}\) in Case~2b, where \(v \in V_P\), \(v+m \notin V_P\), and \(v+m-1 \not\in V_P\). Note that~\(v+m\) can also lie in~\(B_z\).}
  \label{figTreeCutPresDiamCase2b}
\end{figure}

\textbf{Case 2b:} \(v+m-1 \not\in V_P\). \\
First, observe that the cut~\( (M, V\setminus M)\) might cut too many edges in the tree~\(T_z\). Moreover, the cut \( (M \cup T_z', V \setminus (M \cup T_z'))\) cuts few edges, but using \cref{thmTreeDiam} to cut off~\(m\) vertices from~\(G[M \cup T_z']\) might yield a too large bound. The reason for this is that the relative diameter of~\(G[M \cup T_z']\) can be much less than~\(d\), for example when~\(T_z'\) contains~\(\Omega(n)\) vertices and~\(m\) is small. So, instead of using~\(M \cup T_z'\), we will now define a set~\(\tilde{V} \subseteq M \cup T_z'\) such that~\( (\tilde{V}, V \setminus \tilde{V})\) cuts few edges, \(\tilde{V}\) contains all vertices counted by~\(d_P(v,v+m)\), and \(m \leq |\tilde{V}| \leq 2m\), which will ensure that~\(\diam^*(G[\tilde{V}]) \geq \frac{1}{2} d\). 

Let~\( \tilde{m} = 2 | T_z' \cap M |\), which satisfies \(2 \leq \tilde{m} \leq 2 |V(T_z)| -2\) as~\(z \not\in M\). \cref{lemmaApproxCutTree} guarantees an approximate \cut{\tilde{m}}~\((B_z,W_z)\) in~\(T_z\) with~\(z \in W_z\) and~\(e_{T_z} (B_z,W_z) \leq \Delta(G)\).  Define \( \tilde{V} = (M \setminus T_z') \cup B_z\) and note that~\(z \not\in \tilde{V}\). Furthermore, as \( \frac{1}{2} \tilde{m} \leq |B_z| \leq \tilde{m} \) and \( |\tilde{V}| =  m - \frac{1}{2}\tilde{m} + |B_z| \), we have that \(m \leq |\tilde{V}| \leq 2m \). The graph~\(\tilde{G} := G[\tilde{V}]\) consists of at least two components as~\(B_z \neq \emptyset\) and there are no edges between~\(M \setminus T_z'\) and~\(B_z \subseteq T_z'\), see \cref{figTreeCutPresDiamCase2b}. Therefore,
\begin{align*}
  \diam^*( \tilde{G} ) \ \geq \ \frac{d_P(v, v+m) +1 }{| \tilde{V} |} \ \geq \ \frac{\Bfloor{dm}+1}{2m}  \ \geq \ \frac{d}{2}.
\end{align*}
Now, \cref{thmTreeDiam} guarantees a cut~\((\tilde{B},\tilde{W})\) in~\(\tilde{G}\) with~\(|\tilde{B}| =m\) and~\(e_{\tilde{G}} (\tilde{B}, \tilde{W}) \leq \frac{16}{d} \Delta(G)\).  Let~\(B := \tilde{B}\) and \(W := V\setminus B\). Every vertex from~\(P\) that is not counted by~\(d_P(v,v+m)\) is in~\(W\) by construction. Consequently,
  \[ \diam^*(G[W]) \cdot |W| \ \geq \ |V_P| - d_P(v, v+m) \ \geq \ dn- dm,\]
which implies that \( \diam^*(G[W]) \geq d\). 

To estimate the width of the cut~\((B, W)\), consider first the cut~\((\tilde{V}, V \setminus \tilde{V})\). The cut~\( (\tilde{V}, V \setminus \tilde{V})\) cuts at most~\(\Delta(G)\) edges within~\(T_z\). Recall that~\(z \in W_z \subseteq W\). Now, if~\(v\) is the vertex before~\(z\) on~\(P\), then \(\tilde{V} = \{v\} \cup B_z \) and other than the edges in~\(T_z\) only edges incident to~\(v\) are cut by~\((\tilde{V}, V \setminus \tilde{V})\). Otherwise, at most~\(\Delta(G)-1\) edges incident to~\(v\) are cut by~\((\tilde{V}, V \setminus \tilde{V})\) as either~\(v=y_0\) or the edge after~\(v\) on~\(P\) exists and is not cut. Then, from the edges incident to~\(z\) that are cut by~\( (\tilde{V}, V \setminus \tilde{V})\), only the edge before~\(z\) on~\(P\) is not yet counted.  Consequently, \(e_G(\tilde{V}, V \setminus \tilde{V}) \leq 2 \Delta(G)\). Using that~\( (\tilde{B}, \tilde{W})\) cuts at most~\( \frac{16}{d} \Delta(G)\) edges in~\(\tilde{G} = G[\tilde{V}]\) gives the desired bound on the number of cut edges. 

\textbf{Case 3:} \(v \notin V_P\) and \(v+m \in V_P\).\\
This case is similar to Case~2b, but some arguments need to be adjusted as the labeling cannot simply be reversed to obtain a labeling with the same properties due to the requirement that each~\(u\in V_P\) receives the largest label among all vertices in~\(T_u\).  Denote by~\(z\) the path-vertex of~\(v\). As in Case~2b, let~\(\tilde{m} = 2 |T_z' \cap M|\) and let~\((B_z, W_z)\) be an approximate \cut{\tilde{m}} in~\(T_z\) with~\(z \in W_z\) and~\(e_{T_z} (B_z,W_z) \leq \Delta(G)\). For technical reasons, the case when~\(z=v+m\) is treated separately.

\textbf{Case~3a:} \(z= v+m\).\\
Then,~\(d_P(v,v+m) = 0\) and~\(dm <1\).  Define~\(\tilde{V} := B_z \) and~\(\tilde{G} := G[\tilde{V}]\), which satisfy~\(m \leq |\tilde{V}| \leq 2m\) and~\(\diam^*(\tilde{G})  \geq  \frac{1}{2m}  \geq  \frac{1}{2}d\), see \cref{figTreeCutPresDiamCase3a}. Similarly to Case~2b, we can find the desired cut~\((B,W)\) with~\(B \subseteq \tilde{V}\) by applying \cref{thmTreeDiam} to~\(\tilde{G}\).

\begin{figure}
  \begin{subfigure}{0.35\textwidth}
  \begin{center}
    \includegraphics{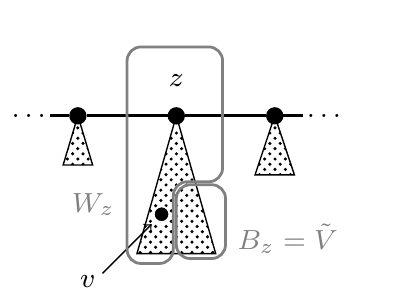}
  \end{center}
  \caption{Case~3a, where~\(z = v+m\).}
  \label{figTreeCutPresDiamCase3a}
  \end{subfigure}
  \begin{subfigure}{0.5\textwidth}
  \begin{center}
    \includegraphics{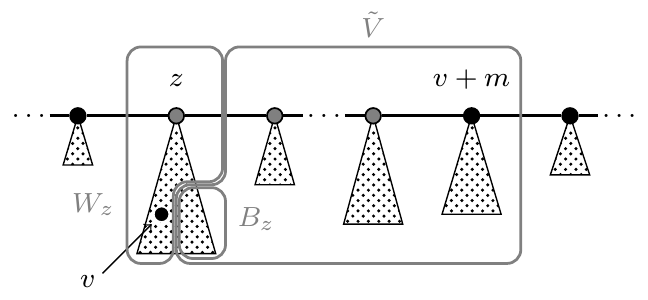}
  \end{center}
  \caption{Case~3b, where~\(z \neq v+m\).}
  \label{figTreeCutPresDiamCase3b}
  \end{subfigure}
  \caption{Proof of \cref{lemmaTreeCutPresDiam}, construction of~\(\tilde{V}\) in Case~3, where~\(v \notin V_P\) and \(v+m \in V_P\). Note that, in both cases,~\(v\) can also lie in~\(B_z\).}
  \label{figTreeCutPresDiamCase3}
\end{figure}

\textbf{Case~3b:} \(z \neq v+m\).\\
Define 
\[\tilde{V} := (M \setminus (T_z' \cup \{z\})) \cup B_z \cup \{v+m\},\]
see \cref{figTreeCutPresDiamCase3b}. This definition of~\(\tilde{V}\) is slightly different than in Case~2b, as here~\(z \in M\) but~\(v+m\) is in~\(\tilde{V}\) instead of~\(z\), which will decrease the bound on the number of cut edges. Now, \(\tilde{V}\) contains exactly~\(d_P(v, v+m)\) vertices of~\(P\). One can argue similar to Case~2b that~\(m \leq |\tilde{V}| \leq 2 m \) as well as that~\(\tilde{G}:=G[\tilde{V}]\) satisfies \(\diam^*(\tilde{G}) \geq \frac{1}{2} d\). Then, \cref{thmTreeDiam} guarantees that there is a cut~\( (\tilde{B}, \tilde{W})\) in~\(\tilde{G}\) with~\( e_{\tilde{G}} (\tilde{B}, \tilde{W}) \leq \frac{16}{d} \Delta(G)\) and~\(|\tilde{B}| = m\). Define~\(B= \tilde{B}\) and~\(W = V\setminus B\). As mentioned before,~\(d_P(v,v+m) = \Bfloor{dm}\) vertices of~\(P\) are in~\(\tilde{V}\). Therefore, at most~\( \Bfloor{dm}\) vertices of~\(P\) are in~\(B\) and, as in Case~2b, it follows that~\(\diam^*(G[W]) \geq d\). Since~\(z \not\in B_z\) and~\(z \not\in \tilde{V}\), it follows that~\(e_G(\tilde{V}, V \setminus \tilde{V}) \leq 2 \Delta(G)\). Hence, the desired bound on~\(e_G(B,W)\) is obtained. This completes the proof of \cref{lemmaTreeCutPresDiam}.

\subsection{Algorithm for Trees}
\label{subsecTreeAlgo}

To achieve the running time in \cref{thmTreeKSec}, it suffices to argue that a cut with the properties claimed by \cref{lemmaTreeCutPresDiam} can be computed in linear time. More precisely, consider a forest~\(G\) on~\(n\) vertices and fix an integer~\(m \in [n]\). The aim is to compute a cut~\((B,W)\) in~\(G\) in~\(\bigO(n)\) time with~\(|B| = m\) that satisfies~\(\diam^*(G[W]) \geq \diam^*(G)\) and~\(e_G(B,W) \leq \left(2 + \frac{16}{\diam^*(G)}\right) \Delta(G)\). The algorithm described here follows the construction presented in \cref{subsecTreeLemma}, which works for trees with maximum degree at least~\(3\). So assume for now that~\(G\) has these properties. 

First, the algorithm computes a longest path~\(P = (v_0, v_1, \ldots, v_{\ell})\) in~\(G\), which takes~\(\bigO(n)\) time\footnote{The following well-known procedure due to Dijkstra computes a longest path in a tree~\(T\). Root~\(T\) in an arbitrary vertex~\(r\) and compute a leaf~\(r'\) at maximum distance from~\(r\). Root~\(T\) in~\(r'\) and compute a leaf~\(r''\) at maximum distance from~\(r'\). Then, the unique \mbox{\(r'\text{,}r''\)-path} in~\(T\) is a longest path in~\(T\).}.  Recall that the ends~\(v_0\) and~\(v_{\ell}\) of~\(P\) were called~\(x_0\) and~\(y_0\) in \cref{subsecTreeLemma}. To compute a \mbox{\(P\)-labeling}, the algorithm rearranges the adjacency list of~\(v_h\) such that~\(v_{h-1}\) is the first entry in the adjacency list of~\(v_h\) for all~\(h \in [\ell]\). Then, the algorithm traverses the tree~\(G\) with a depth-first search starting in~\(v_{\ell} = y_0\) and labels each vertex when it turns black (i.e., when the processing of the vertex finishes, see the notation in~\cite{CLRS09}). It is easy to see that the computation of the \mbox{\(P\)-labeling} takes time proportional to~\(n\) and that it can be stored such that converting between vertices and labels and vice versa takes constant time.  In the implementation, we do not change the vertex names or identify vertices with their labels as in \cref{subsecTreeLemma}.  Furthermore, while computing the \mbox{\(P\)-labeling}, the algorithm computes~\(d_P(x_0,x)\) for each vertex~\(x \in V\), i.e., the number of vertices of~\(P\) that are already labeled right before~\(x\) is labeled. Now, \(d_P(x,y) = d_P(x_0,y) - d_P(x_0,x)\) for all~\(x,y \in V\) where~\(x \leq y\) and a vertex~\(v \in V\) with~\(d_P(v, v+m) =\Bfloor{\diam^*(G) \cdot m}\) and~\(v \in V_P\) or~\(v+m \in V_P\) can be computed in~\(\bigO(n)\) time. Using that the algorithms contained in \cref{thmTreeDiam} and \cref{lemmaApproxCutTree} take linear time, it follows that a cut with the desired properties can be constructed in~\(\bigO(n)\) time.  For more details see Chapter~6.2 in~\cite{thesisTina}.

To conclude, consider the case when~\(G\) is not a tree with~\(\Delta(G) \geq 3\).  Clearly, if~\(\Delta(G) \leq 2\), a cut in~\(G\) with the desired properties can be computed in~\(\bigO(n)\) time. If~\(G\) is not connected and~\(\Delta(G) \geq 3\), the algorithm adds edges to~\(G\) until a tree~\(T\) with~\(\diam^*(T) = \diam^*(G)\) and~\(\Delta(T) = \Delta(G)\) is obtained. More precisely, these additional edges need to join the ends of two longest paths in different components of~\(G\). As each connected component~\(G'\) of~\(G\) is a tree, a longest path in~\(G'\) can be computed in~\(\bigO(|V(G')|)\) time as mentioned above. Since~\(e_G(B,W) \leq e_T(B,W)\) holds for every cut~\((B,W)\) in~\(T\), the procedure described above can be applied to~\(T\) to obtain the desired cut in~\(G\).

\subsection{Remarks on Improving \cref{thmTreeKSec}}
\label{subsecTreeRemarks}

To improve the bound on the cut width in \cref{thmTreeKSec} as stated in \eqref{eqTreeKSecImpr}, recall that the proof of \cref{lemmaTreeCutPresDiam} uses \cref{thmTreeDiam} to estimate the width of the cut in~\(\tilde{G}\). The bound on the cut width in \cref{thmTreeDiam} can be improved to 
\begin{equation*}
  e_G(B,W) \ \leq \ \frac{1}{2} \left( \left( \log_2 \left( \frac{1}{\diam^*(G)} \right) \right)^2 + 7 \log_2 \left(\frac{1}{\diam^*(G)} \right) +6 \right) \Delta(G), 
\end{equation*}

see Section~1.4 in \cite{PaperGenTreeDiam} or Theorem~5.12 in \cite{thesisTina}. Using this, the bound on the width of the cut in \cref{lemmaTreeCutPresDiam} improves to
\begin{align*}
  e_G(B,W) \ \leq \; \; & 2 \Delta(G) + \ \frac{1}{2} \ \left( \left( \log_2 \left( \frac{2}{\diam^*(G)} \right) \right)^2 + 7 \log_2 \left( \frac{2}{\diam^*(G)} \right) +6  \right) \Delta(G) \nonumber \\
  = \; \; & \frac{1}{2} \ \left( \left( \log_2 \left( \frac{1}{\diam^*(G)} \right) \right)^2 + 9 \log_2 \left( \frac{1}{\diam^*(G)} \right) +18  \right) \Delta(G) 
\end{align*}
and the improvement on the bound on the width of the \ksec{k} in \cref{thmTreeKSec} as stated in \eqref{eqTreeKSecImpr} follows.

Last but not least for \ksec{k}s in trees, we mention that it does not seem to be obvious whether our algorithm or its analysis can be modified to obtain a linear time algorithm. More precisely, it is not clear how to reduce the dependency on \(k\) in the running time.

\section{Minimum $k$-Section in Tree-Like Graphs}
\label{secProofGen}

This section concerns the proof of \cref{thmGenKSec} about \ksec{k}s in tree-like graphs. We begin with presenting the definition and some facts about tree decompositions as well as some tools for tree-like graphs in \cref{subsecPrelTreeDec}. As in the proof of \cref{thmTreeKSec}, \cref{subsecGenProofThm} presents a lemma that immediately implies the existence part of \cref{thmGenKSec} and is proved in \cref{subsecGenProofLemma}. The main idea is similar to the proof in \cref{subsecTreeLemma}. Hence, \cref{subsecGenProofLemma} focuses on the aspects that are more involved than in the case of trees and refers to \cref{subsecTreeLemma} for steps that are analogous to the case of trees. All algorithmic aspects of \cref{thmGenKSec} are presented in \cref{subsecGenProofThm}. A detailed proof of \cref{thmGenKSec} can be found in Chapter~6.3 in~\cite{thesisTina}.

\subsection{Preliminaries for Tree Decompositions}
\label{subsecPrelTreeDec}

Let us start by recalling the definition of a tree decomposition.

\begin{defi}\label{defTreeDec}
  Let~\(G\) be a graph, \(T\) be a tree, and~\(\X= (X^i)_{i \in V(T)}\) with~\(X^i \subseteq V(G)\) for each~\({i \in V(T)}\). The pair~\((T,\X)\) is a \emph{tree decomposition} 
  of~\(G\) if the following three properties 
  hold.
  \begin{enumerate}[(T1), leftmargin= \IdentationTDConditions]
    \item For every~\(v \in V(G)\), there is an~\(i \in V(T)\) such that~\(v \in X^i\).
    \item For every~\(e \in E(G)\), there is an~\(i \in V(T)\) such that~\(e \subseteq X^i\). 
    \item For all~\(i, j \in V(T)\) and all~\(h \in V(T)\) on the (unique)~\(i\text{,}j\)-path in~\(T\), we have~\(X^i \cap X^j \subseteq X^h\). 
  \end{enumerate}  
  The \emph{width} of~\((T,\X)\) is defined as \(\max \{|X^i|-1 \colon \, i \in V(T)\}\). 
  The \emph{tree-width} of~\(G\), denoted by~\(\tw(G)\), is the smallest integer~\(t\) such that~\(G\) allows a tree decomposition of width~\(t\). 
\end{defi}

Consider a graph~\(G=(V,E)\) and a tree decomposition~\( (T,\X)\) with~\( \X= (X^i)_{i\in V(T)}\) of~\(G\). To distinguish the vertices of~\(G\) from the vertices of~\(T\) more easily, we refer to the vertices of~\(T\) as \emph{nodes}. Furthermore, for~\(i \in V(T)\), we refer to the set~\(X^i\) as the \emph{cluster} of~\(\X\) that corresponds to~\(i\), or simply the cluster of~\(i\) when the tree decomposition is clear from the context. It is easy to show that (T3) is equivalent to the following condition.
\begin{enumerate}[\emph{(T3')}, leftmargin= \IdentationTDExtraCondition]
 \item \emph{For every~\(v \in V\), the graph~\(T[I_v]\) is connected, where~\(I_v := \left\{i \in V(T) \colon \, v \in X^i \right\}\).}
\end{enumerate}

Consider a graph~\(G_0\) and a tree decomposition~\((T_0,\X_0)\) with~\(\X_0 = (X_0^i)_{i \in V(T_0)}\). In order to apply a procedure to a subgraph~\(G \subseteq G_0\), it is often necessary to construct a tree decomposition of~\(G\). For this purpose, the \emph{tree decomposition~\((T,\X)\) induced by~\(G\) in~\((T_0,\X_0)\)} is defined by~\(T=T_0\) and~\(X^i = X_0^i \cap V(G)\) for all~\(i \in V(T)\), where~\(\X = (X^i)_{i \in V(T)}\). Observe that~\((T,\X)\) is indeed a tree decomposition of~\(G\) as well as that the width and the size of~\((T,\X)\) are at most the width and the size of~\((T_0,\X_0)\), respectively. Usually, some clusters of an induced tree decomposition are empty, which can be avoided with the following concept. A tree decomposition~\( (T,\X)\) of a graph~\(G\) with~\( {\X = (X^i)_{i \in V(T)} } \) is called~\emph{nonredundant} if~\(X^i \not \subseteq X^j\) and~\(X^j \not \subseteq X^i\) for every edge~\( \{i,j\}\) in~\(T\). The next proposition says that any tree decomposition can be transformed into a nonredundant one without increasing its width in linear time. Recall that the size~\(\|(T,\X)\|\) and the relative weight of a heaviest path~\(r(T,\X)\) were introduced in \cref{subsecResults} before stating \cref{thmGenKSec}. 

\begin{prop}[Proposition~20 in \cite{PaperGenTreeDiam}]
  \label{propTDNonRed}
  For every tree decomposition~\((T,\X)\) of a graph~\(G\) with \(V(G) = [n]\) for some~\(n \in \mathbb{N}\), a nonredundant tree decomposition~\((T',\X')\) of~\(G\) of the same width as~\((T,\X)\) that satisfies~\(\|(T',\X')\| \leq \|(T,\X)\|\) and~\(r(T',\X') \geq r(T,\X)\)  can be computed in~\(\bigO(\|(T,\X)\|)\) time. 
\end{prop}

When working with a tree~\(\tilde{T}\), for example in the proof of \cref{lemmaApproxCutTree} or \cref{lemmaTreeCutPresDiam}, the following cuts were applied. For some vertex~\(v \in V(\tilde{T})\) we removed all edges incident to~\(v\) and combined the vertex sets of the resulting connected components to obtain a cut in~\(\tilde{T}\). Each time such a construction was used, at most~\(\deg(v) \leq \Delta(\tilde{T})\) edges were cut. This can be generalized by considering clusters of a tree decomposition, as done in the next lemma. It uses the following notation: Consider a graph~\(G\) and a tree decomposition~\((T,\X)\) of~\(G\). For each node~\(i\) in~\(T\), we denote by~\(E_{G}(i)\) the set of edges~\(e \in E(G)\) such that~\(e \cap X^i \neq \emptyset\), where~\(X^i\) is the cluster of~\(i\). Note that~\(|E_G(i)| \leq t \Delta(G)\) for every~\(i \in V(T)\), where~\(t-1\) denotes the width of~\( (T,\X)\).  We say that two subgraphs~\(H_1\subseteq G\) and~\(H_2 \subseteq G\) are \emph{disjoint parts} of~\(G\) if~\(V(H_1) \cap V(H_2)= \emptyset\) and there is no edge~\(e=\{x,y\}\) in~\(G\) with~\(x \in V(H_1)\) and~\(y \in V(H_2)\). Note that, if~\(G\) is not connected, then two distinct connected components of~\(G\) are disjoint parts of~\(G\), but the subgraph~\(H_i\) for~\(i \in \{1,2\}\) in the definition of disjoint parts does not have to be connected. 

\begin{lemma}[see Fact~10.13 and Fact~10.14 in \cite{KleinbergTardos} or Corollary~1.8 in \cite{ReedBook}]
  \label{lemmaRemoveClusterEdges}
  Let~\(G=(V,E)\) be an arbitrary graph and let~\( (T, \X)\) be a tree decomposition of~\(G\) with~\({ \X = (X^i)_{i \in V(T)} } \). Fix some node~\(i \in V(T)\), let~\(\ell := \deg_T (i) \) and denote by~\( i_1, i_2, \ldots, i_{\ell}\) the neighbors of~\(i\) in~\(T\). Furthermore,  for~\( h \in [\ell]\), let~\(V_{h}^T\) be the node set of the component of~\(T-i\) that contains~\(i_h\). Removing the edges in~\(E_G(i)\) from~\(G\) decomposes~\(G\) into~\(\ell + |X^i| \) disjoint parts, which are~\( (\{v\}, \emptyset) \) for every~\(v \in X^i\) and~\(G[V_h] \) for every~\( h \in [\ell]\), where~\( V_h \ := \bigcup_{j \in V_h ^T} X^j \setminus X^i\).
\end{lemma}

This lemma says that if we remove the edges in~\(E_G(i)\) for some~\(i \in V(T)\), then the graph~\(G\) splits into several disjoint parts. Hence we can combine these disjoint parts in an arbitrary way to obtain a cut in~\(G\) of width at most~\(t \Delta(G) \). This idea suffices to generalize the existence part of \cref{lemmaApproxCutTree} to arbitrary graphs with a given tree decomposition. The algorithmic part is a direct consequence of Lemma~4 in \cite{PaperGenTreeDiam}. 

\begin{lemma}
  \label{lemmaApproxCutGen} Let~\(G\) be an arbitrary graph on~\(n\) vertices and let~\((T,\X)\) be a tree decomposition of~\(G\) of width at most~\(t-1\). For every integer~\(m \in [2n]\), there is an approximate \mbox{\(m\)-cut}~\((B,W)\) in~\(G\) with~\(e_G(B,W) \leq t \Delta(G)\). If the tree decomposition~\((T,\X)\) is provided as input and~\(V(G)=[n]\), then a cut satisfying these requirements can be computed in~\(\bigO( \|(T,\X)\|)\) time. 
\end{lemma}

Furthermore, \cref{thmTreeDiam} (or, more precisely, its improved version mentioned in \cref{subsecTreeRemarks}) can be generalized to arbitrary graphs. Recall that, instead of working with the relative diameter, we now use the relative weight of a heaviest path in a given tree decomposition, which means the following.  Consider a tree decomposition~\( (T,\X)\) of some graph~\(G\) with~\(\X = (X^i)_{i \in V(T)}\) and a path~\(P \subseteq T\). The \emph{weight} of~\(P\) with respect to~\(\X\) is~\(w_{\X}(P) := \left| \bigcup_{i \in V(P)} X^i \right|\) and the \emph{relative weight} of~\(P\) with respect to~\(\X\) is~\(w^*_{\X} (P) = \frac{1}{n} w_{\X}(P)\), where~\(n\) denotes the number of vertices of~\(G\). Furthermore, in \cref{subsecResults},~\(r(T,\X)\) was defined to be the \emph{relative weight of a heaviest path} in~\(T\), i.e., \(r(T,\X) = w^*_{\X}(P^*)\) where~\(P^*\subseteq T\) is a path with~\(w_{\X} (P^*) \geq w_{\X}(P)\) for all paths~\(P \subseteq T\). 

\begin{thm}[similar to Theorem~3 in~\cite{PaperGenTreeDiam}]
  \label{thmGenTreeDiam}
  Let~\(G\) be an arbitrary graph on~\(n\) vertices and let~\((T,\X)\) be a tree decomposition of~\(G\) of width at most~\(t-1\). For every integer~\(m \in [n]\), there is a cut~\((B,W)\) in~\(G\) with~\(|B| =m\) that satisfies
  
  \[ e_G(B,W) \ \leq \ \frac{t}{2}  \left( \left( \log_2  \frac{1}{r(T,\X)} \right)^2 \ + \  9 \log_2  \frac{1}{r(T,\X)} \ + \ 8 \right)\ \Delta(G).\] 
  If the tree decomposition~\( (T,\X)\) is provided as input and~\(V(G)=[n]\), a cut~\((B,W)\) with these properties can be computed in~\( \bigO \left(  \| (T, \X) \| \right) \) time. 
\end{thm}

\subsection{Proof for \cref{thmGenKSec}}
\label{subsecGenProofThm}

The main idea of the proof of \cref{thmGenKSec} is similar to the proof of \cref{thmTreeKSec}, i.e., the pieces~\(V_{\ell}\) of a \ksec{k}~\((V_1, \ldots, V_{k})\) are cut off successively from the graph~\(G\) while ensuring that the relative weight of a heaviest path in the tree decomposition induced by the remaining part is at least as large as the relative weight of a heaviest path in the original tree decomposition.  The next lemma states this formally. 

\begin{lemma}
  \label{lemmaCutPresR}
  Let~\(G\) be an arbitrary graph on~\(n\) vertices and let~\( (T,\X)\) be a tree decomposition of~\(G\) of width~\(t-1\). For every integer~\(m \in [n-1]\), there is a cut \( (B,W)\) in~\(G\) with~\(|B| = m\), 
  \[ e_{G}(B,W) \ \leq \ \frac{t}{2}  \left( \left( \log_2 \left( \frac{1}{r(T,\X)} \right) \right)^2 + 11 \log_2 \left( \frac{1}{r(T,\X)} \right) + 24 \right) \Delta(G),\]
  and such that~\( r (T',\X')  \geq r(T,\X) \) holds for the tree decomposition~\( (T',\X')\) induced by~\(G[W]\) in~\( (T,\X)\). 
\end{lemma}

The existence part of \cref{thmGenKSec} follows immediately from the previous lemma. 

\subsection{Proof of \cref{lemmaCutPresR}}
\label{subsecGenProofLemma}

For the remaining section, fix an arbitrary graph~\(G=(V,E)\) on~\(n\) vertices and let~\((T,\X)\) with~\(T=(V_T, E_T)\) and~\(\X =(X^i)_{i \in V_T}\) be a tree decomposition of~\(G\). Due to \cref{propTDNonRed} we may assume that~\((T,\X)\) is nonredundant. Define~\(r := r(T,\X)\) and denote by~\(t-1\) the width of~\((T,\X)\). Furthermore, fix a path~\(P \subseteq T\) of relative weight~\(r\) with respect to~\((T,\X)\) and let~\(P=(V_P, E_P)\). 

\subsubsection{Notation and Vertex Labeling}
\label{subsubsecGenNotation}

First, we settle some notation and introduce a labeling similar to the labeling used in \cref{subsecTreeLemma}. Fix one end~\(i_0\) of~\(P\). Consider two neighboring nodes~\(i\) and~\(j\) on~\(P\). We say that~\(i\) is the \emph{node before~\(j\) on~\(P\)}, if~\(i\) is passed before~\(j\) when traversing~\(P\) from~\(i_0\) to its other end, say~\(j_0\), and otherwise we say that~\(i\) is the \emph{node after~\(j\) on~\(P\)}. For technical reasons, this notion is extended to the nodes~\(i_0\) and~\(j_0\) by saying that~\(i_0\) is the node after~\(j_0\) on~\(P\) and that~\(j_0\) is the node before~\(i_0\) on~\(P\). 

Define~\(R:= \bigcup_{i \in V_P} X^i\) and~\(S:= V \setminus R\). For each~\(i \in V_P\),  denote by~\(T_i\) the component of~\(T- E_P\) that contains~\(i\). Moreover, for each~\(x \in R\), the unique node~\(i \in V_P\) that is closest to~\(i_0\) among all nodes~\(j \in V_P\) with~\(x \in X^j\) is called the \emph{path-node} of~\(x\). Define
\[ R_i := \{x \in X^i\colon\, \text{\(i\) is the path-node of \(x\)}\} \quad \quad \text{and}\quad \quad S_i := \bigcup_{j \in V(T_i)} X^j \setminus R\]
for all~\(i \in V_P\). For~\(x\in S\), the node~\(i \in V_P\) is called the \emph{path-node} of~\(x\) if and only if~\(x \in S_i\). The next proposition lists some properties of these sets and implies that every vertex~\(x \in V\) has a unique path-node~\(i \in V_P\).

\begin{prop}
  \label{propPropRiSi}
  \begin{enumerate}[a)]
    \item[]
    \item\label{propPropRiSiPartition} The sets~\(R_i\) with~\(i \in V_P\) form a partition of~\(R\) and the sets~\(S_i\) with~\(i \in V_P\) form a partition of~\(S\).
    \item\label{propPropRiSiNonEmpty} For each~\(i \in V_P\), the set~\(R_i\) is not empty. 
  \end{enumerate}
\end{prop}

\begin{proof}
  \begin{enumerate}[a)]
    \item[]
    \item The statement for the sets~\(R_i\) is obvious, the statement for the sets~\(S_i\) follows from~(T3').
    \item Since~\((T,\X)\) is nonredundant,~\(X^{i_0} \neq \emptyset\) and, if~\(V_P \neq \{i_0\}\), also~\(X^i \not \subseteq X^{i^-}\) for all~\(i \in V_P \setminus \{i_0\}\), where~\(i^-\) denotes the node before~\(i\) on~\(P\). Hence, \(R_i \neq \emptyset\) for all~\(i \in V_P\). \qedhere
  \end{enumerate}
\end{proof}

The sets~\(R_i\) and the nodes on the path~\(P\) both correspond to the vertices in the path~\(P\) in the proof of \cref{lemmaTreeCutPresDiam}: \(R_i\) is a subset of the vertices of~\(G\) and~\(V_P\) is a set of nodes of~\(T\). Similarly, the vertex sets~\(S_i\) and the node sets~\(V(T_i)\setminus\{i\}\) both correspond to the sets~\(T_v'\) in the proof of \cref{lemmaTreeCutPresDiam}.

A \emph{\(P\)-labeling}\footnote{We use the same term as in \cref{secProofTree} as it will be clear from the context whether a tree or a graph with a given tree decomposition is considered.} of~\(G\) with respect to~\( (T,\X)\) is a labeling of the vertices in~\(V\) with~\( \{1,2,\ldots, n\}\), such that
\begin{itemize}
  \item for each node~\(i \in V_P\), the vertices of~\(R_i \cup S_i\) receive consecutive labels and the vertices in~\(R_i\) receive the largest labels among those, and
  \item for all nodes~\(i,j \in V_P\) with~\(i \neq j\), if~\(i_0\) is closer to~\(i\) than to~\(j\), then each vertex in~\(R_i \cup S_i\) has a smaller label than every vertex in~\(R_j \cup S_j\). 
\end{itemize}
From now on, fix a \mbox{\(P\)-labeling} and identify each vertex with its label. As in \cref{subsecTreeLemma}, any number that differs by a multiple of~\(n\) from a label in~\([n]\) is considered to be the same as that label and when comparing vertices we always refer to their labels in~\([n]\). Moreover, the notion of \(a\) is between~\(b\) and~\(c\) from \cref{subsecTreeLemma} is adapted. The labeling is useful for finding certain cuts related to the sets~\(R_i\) and~\(S_i\) in the graph~\(G\). This is made precise by the next proposition, which is a direct consequence of \cref{lemmaRemoveClusterEdges}.

\begin{prop}
  \label{propCutPlabeling}
  Let~\(i\) be an arbitrary node in~\(P\), and denote by~\(i^{-}\) and~\(i^{+}\) the nodes before and after~\(i\) on~\(P\), respectively. Let~\(x^{-}\) be the vertex with the largest label in~\(R_{i^{-}}\) and let~\(x^{+}\) be the vertex with the smallest label in~\(S_{i^{+}} \cup R_{i^{+}}\). Moreover, if~\(i=i_0\), let~\(V_P^+ = V_P \setminus \{i_0\}\); if~\(i = j_0\), let~\(V_P^{-} = V_P\setminus\{j_0\}\); and otherwise let~\(V_P^-\) and~\(V_P^+\) be the node sets of the connected components of~\(P-i\), that contain~\(i^{-}\) and~\(i^{+}\), respectively. 
  Removing from~\(G\) the edges~\(E_G(i)\) decomposes~\(G\) into the following disjoint parts
  \begin{itemize}[leftmargin= \IdentationTDConditions]
      \item an isolated vertex for each~\(v \in R_i\),
      \item if~\(S_i \neq \emptyset\), the part~\(G[S_i]\),
      \item if~\( i \neq i_0\), the subgraph of~\(G\) induced by~\(\bigcup_{j \in V_P^-} (R_j \cup S_j) = \{1, \ldots, x^-\}\), and
      \item if~\( i \neq j_0\), the subgraph of~\(G\) induced by~\(\bigcup_{j \in V_P^+} (R_j \cup S_j) = \{x^+, \ldots, n\}\).
    \end{itemize}
\end{prop}

\subsubsection{Construction of the Black Set}
\label{subsubsecGenCases}

The idea for the proof of \cref{lemmaCutPresR} is similar to its tree version, namely the proof of \cref{lemmaTreeCutPresDiam}. One difference is that, instead of cutting along single edges of the decomposition tree, we will work with cuts arising from the removal of entire clusters from the graph. This is also reflected by the slightly different polylogarithmic terms in the bounds in \eqref{eqTreeKSecImpr} and \cref{thmGenKSec}. Again, we will define a set~\(\tilde{V}\) that contains enough vertices to form the desired set~\(B\) by applying \cref{thmGenTreeDiam} to a graph~\(\tilde{G}\) with~\(V(\tilde{G}) = \tilde{V}\) and a suitable tree decomposition~\( (\tilde{T}, \tilde{\X})\). In the proof of \cref{lemmaTreeCutPresDiam}, the forest induced by the set~\(\tilde{V}\) was not connected and, hence, it was easy to take care of rounding effects concerning the relative diameter of~\(G[\tilde{V}]\). Now, when working with tree decompositions, the graph induced by~\(\tilde{V}\) might be connected and it requires more work to reorganize the tree decomposition.  More precisely, we will artificially disconnect the graph~\(G[\tilde{V}]\),  and then glue two tree decompositions of subgraphs of~\(G[\tilde{V}]\) together to obtain a tree decomposition~\((\tilde{T}, \tilde{\X})\) of~\(\tilde{G}\) with~\(r(\tilde{T}, \tilde{\X}) \geq \frac{1}{2}r(T,\X)\).

Fix an arbitrary integer~\(m \in [n-1]\) and observe that~\( |R| = rn\). For two vertices~\(x,y \in V\),  define the \emph{\(R\)-distance} of~\(x\) and~\(y\) as
 \[ d_R(x,y) = \left| \left\{ v \in R \setminus \{y\}\colon\, \vphantom{'}\text{\(v\) is between~\(x\) and~\(y\)} \right\} \right|.\]
Analogously to finding the vertex~\(v\) in \cref{subsecTreeLemma}, we can argue that there is a vertex~\(v \in V\) with \(d_R (v,v+m) = \Bfloor{rm}\) and  \(v \in R\) or~\(v+m \in R\). Define 
  \[ M := \{u \in V\colon\, u \text{ is between~\(v\) and~\(v+m -1\)}\} , \]
and note that~\(|M| = m\). 

In the following figures, the tree~\(T\) is drawn in the top and the vertex sets containing vertices of the graph~\(G\) are drawn underneath the corresponding node of~\(P\). More precisely, the path~\(P\subseteq T\) is drawn explicitly on the top and, for each~\(h \in V_P\), the node~\(h\) is drawn in black and the tree~\(T_h\) is indicated by a triangle. Furthermore, for each~\(h \in V_P\), the sets~\(R_h\) and~\(S_h\) are represented by a circle and a trapezoid, respectively, and are drawn underneath the node~\(h\). Areas that are colored gray inside a set~\(R_h\) visualize that some vertices of~\(R_h\) are counted by~\(d_R(v, v+m)\). 

\textbf{Case 1:} \(v \in R\) and \(v+m \in R\). \\
Define~\(B:=M\) and \(W := V \setminus B\). Due to \cref{propCutPlabeling}, we have~\(E_G(B,W) \subseteq E_G(i) \cup E_G(j)\) and the cut~\((B,W)\) satisfies the desired bound on its width, see \cref{figCutPresRCase1}.  Let~\( (T', \X')\) be the tree decomposition induced by~\(G[W]\) in~\( (T,\X)\). Then, 
  \[ r(T', \X') \ \geq \ \frac{w_{\X'}(P)}{|W|} \ = \ \frac{|R| - d_R(v, v+m)}{|W|} \ \geq \ \frac{ rn - rm }{n-m} \ = \ r ,\]
as desired.

\begin{figure}
  \begin{subfigure}{0.48 \textwidth}
  \begin{center}
    \includegraphics{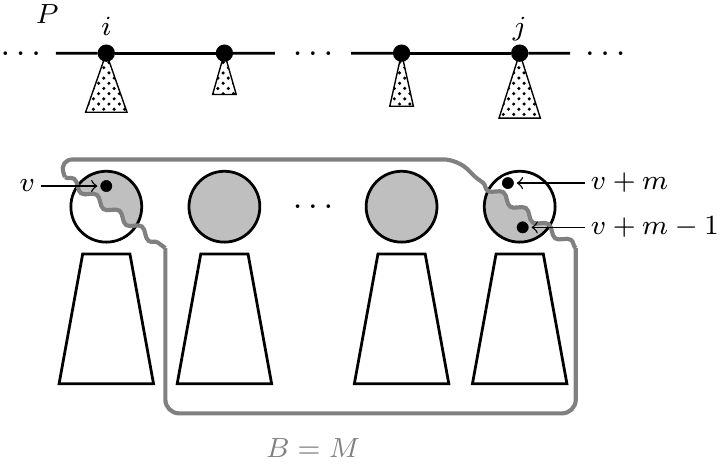}
  \end{center}
  \caption{Case 1, where~\(v \in R\) and~\(v+m \in R\).}
  \label{figCutPresRCase1}
  \end{subfigure}
  \begin{subfigure}{0.48 \textwidth}
  \begin{center}
    \includegraphics{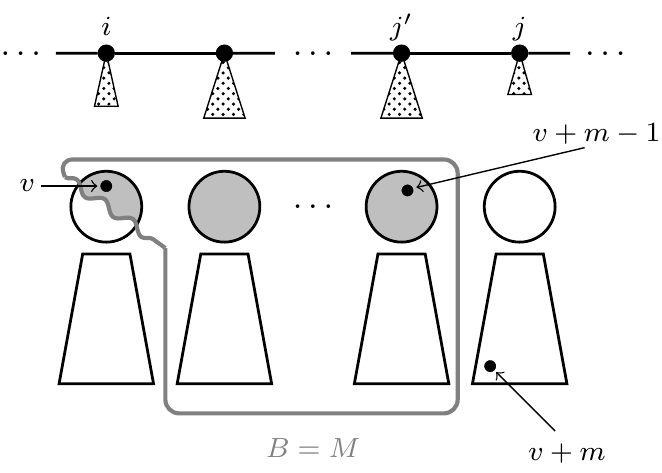}
  \end{center}
  \caption{Case 2a, where~\(v \in R\) and~\(v+m-1 \in R\).}
  \label{figCutPresRCase2a}
  \end{subfigure}
  
  \caption{Construction of the black set in Case~1 and Case~2a.}
  \label{figCutPresRCase1and2a}
\end{figure}
 
\textbf{Case 2:} \(v \in R\) and \(v+m \not\in R\). \\
As in \cref{subsecTreeLemma}, the case when~\(v+m-1 \in R\) is treated separately for technical reasons. Let~\(j'\) be the node before~\(j\) on~\(P\). 

\textbf{Case~2a:} \(v+m-1 \in R\). \\
Similarly to Case~1, the cut~\((B,W)\) with~\(B:=M\) and~\(W:=V\setminus M\) satisfies all requirements, see also \cref{figCutPresRCase2a}.

\textbf{Case 2b:} \(v+m-1 \not\in R\). \\
Define \( \tilde{m} := 2 |S_j \cap M|\), which satisfies~\(2 \leq \tilde{m} \leq 2m\) as~\(v+m-1 \in S_j\) due to \cref{propPropRiSi}\ref{propPropRiSiNonEmpty}. \cref{lemmaApproxCutGen} guarantees an \mbox{\(\tilde{m}\)-approximate} cut~\((B_j,W_j)\) in~\(G[S_j]\) with~\(e_{G[S_j]}(B_j, W_j) \leq t \Delta(G)\), because the induced tree decomposition of~\(G[S_j]\) with respect to~\((T,\X)\) has width at most~\(t-1\). Then, the set \( \tilde{V} := ( M \setminus S_j) \cup B_j \) satisfies \(  |\tilde{V} |  =  m - \tfrac{1}{2}\tilde{m} + |B_j| \) and 
\begin{align}
  \label{proofGenCase2}
  m \ \leq \ | \tilde{V}| \  \leq \ m + \tfrac{1}{2} \tilde{m} \ \leq \ 2m,
\end{align}
see also \cref{figCutPresRCase2b}. Note that~\(\tilde{V}\) might contain vertices from~\(X^j\), the cluster of node~\(j\), and, hence,~\(G[\tilde{V}]\) might be connected. Let~\(\tilde{G}\) be the graph obtained from~\(G[\tilde{V}]\) by removing all edges in~\(E_G(j)\) and observe that~\(e_{\tilde{G}} (B_j, \tilde{V} \setminus B_j) = \emptyset\) due to \cref{propCutPlabeling}. Denote by~\( (\tilde{T}_1, \tilde{\X}_1 ) \) and~\( (\tilde{T}_2, \tilde{\X}_2 ) \) the induced tree decompositions of~\(\tilde{G}[\tilde{V} \setminus B_j]\) and~\(\tilde{G}[B_j]\) with respect to~\((T,\X)\), respectively. Furthermore, let~\(\tilde{P}_1 = P\) and let~\(\tilde{P}_2\) be a path in~\((\tilde{T}_2, \tilde{\X}_2)\) that consists of one node~\(h_0\) whose cluster in~\(\tilde{\X}_2\) is non-empty. Then,~\(w_{\tilde{\X}_1}(\tilde{P}_1) \geq d_R(v, v+m)\) and~\(w_{\tilde{\X}_2}(\tilde{P}_2) \geq 1\). Now, define~\( \tilde{T} \) to be the tree obtained from taking one copy of~\(\tilde{T}_1\) and  one copy of~\(\tilde{T}_2\) with disjoint node sets and adding an edge between~\(j_0\) in~\(\tilde{T}_1\) and~\(h_0\) in~\(\tilde{T}_2\). Denote by~\(\tilde{\X}\) the corresponding union of~\(\tilde{\X}_1\) and~\(\tilde{\X}_2\). Then, \( (\tilde{T}, \tilde{\X})\) is a tree decomposition of~\( \tilde{G}\) of width at most~\(t-1\) with
\[ r(\tilde{T}, \tilde{\X}) \ \geq \ \frac{w_{\tilde{\X}_1}(\tilde{P}_1) + w_{\tilde{\X}_2}( \tilde{P}_2 )}{|\tilde{V}|} \ \geq \ \frac{d_R (v,v+m) + 1}{2m} \ \geq \ \tfrac{1}{2}r\]
due to~\eqref{proofGenCase2}. Therefore, \cref{thmGenTreeDiam} implies that~\(\tilde{G}\) allows a cut~\((\tilde{B}, \tilde{W})\) with~\(|\tilde{B}| = m\) and 
\begin{align}
  \label{proofGenWidthTilde} 
  e_{\tilde{G}}(\tilde{B}, \tilde{W}) \ \leq \ \frac{t}{2}  \left( \left( \log_2 \left( \frac{1}{r} \right) \right)^2 + 11 \log_2 \left( \frac{1}{r} \right) + 18 \right)\Delta(G).
\end{align}

\begin{figure}
  \begin{subfigure}{0.48 \textwidth}
  \begin{center}
    \includegraphics{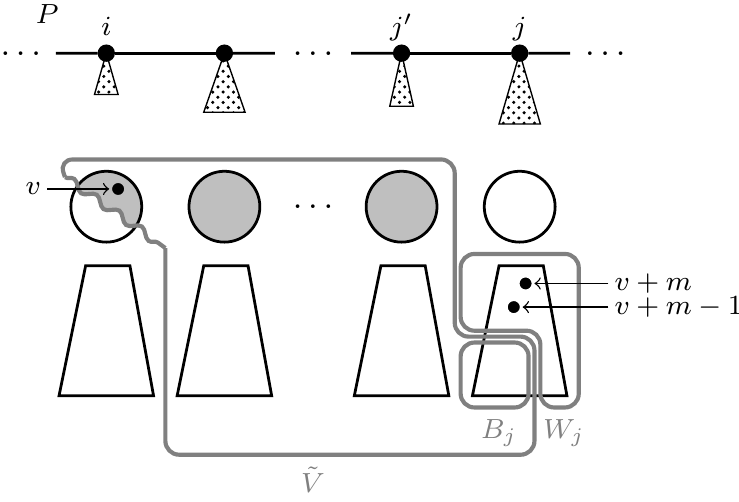}
  \end{center}
  \caption{Case~2b, where~\(v \in R\) and \(v+m, v+m-1 \notin R\). Note that~\(v+m\) and~\(v+m-1\) might also be in the set~\(B_j\).}
  \label{figCutPresRCase2b}
  \end{subfigure}
  \hspace{1em}
  \begin{subfigure}{0.48 \textwidth}
  \begin{center}
    \hspace*{-1em}\includegraphics{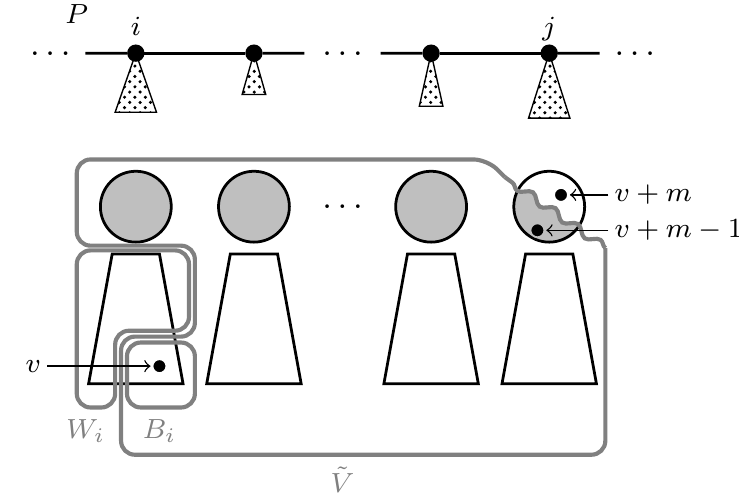}
  \end{center}
  \caption{Case~3, where \(v \not\in R\) and \(v+m \in R\). Note that~\(v+m-1\) might also be in~\(S_j\) and~\(v\) might also be in~\(W_i\).}
  \label{figCutPresRCase3}
  \end{subfigure}
  
  \caption{Construction of the black set in Case~2b and Case~3.}
  \label{figCutPresRCase2band3}
\end{figure}

Now, define~\(B := \tilde{B}\) and~\(W :=  V\setminus \tilde{B} \). Furthermore, denote by~\( (T',\X')\) the tree decomposition induced by~\(G[W]\) in~\( (T,\X)\). By construction, there are exactly~\(d_R(v, v+m)\) vertices of~\(R\) in~\(\tilde{V}\) and, hence, at least~\(|R| - d_R(v, v+m) \geq rn-rm\) vertices of~\(R\) are in~\(W\) and
\[ r(T',\X') \ \geq \  \frac{ w_{\X'}(P) }{|W|} \ \geq \ \frac{|R \cap W|}{|W|} \ \geq \ \frac{ r(n - m) }{n-m}  \geq \ r.\]

Next, the width of the cut~\((B,W)\) in~\(G\) is analyzed. Let~\( \hat{G} := G- E_G(i) - E_G(j)\), which contains no edges between the sets~\(M\setminus S_j\), \(S_j\), and~\(V\setminus(M \cup S_j)\) due to \cref{propCutPlabeling}. Therefore,
\[e_G(\tilde{V},V\setminus \tilde{V}) \ \leq \ 2 t \Delta(G)  + e_{\hat{G}} (\tilde{V},V \setminus \tilde{V}) \ \leq \ 3 t \Delta(G),\]
where the previous estimation also counts all edges that were removed from~\(G[\tilde{V}]\) when constructing~\(\tilde{G}\). Now,~\eqref{proofGenWidthTilde} yields the desired bound on the width of~\((B,W)\).

\textbf{Case 3:} \(v \not\in R\) and \(v+m \in R\). \\
This case is similar to Case~2b, but not completely analogous. Instead of splitting~\(S_j = B_j \cup W_j\), the set~\(S_i\) is split now. To do so, define~\(\tilde{m} := 2 |S_i \cap M|\), which satisfies~\(2 \leq \tilde{m} \leq 2m\). \cref{lemmaApproxCutGen} implies that there is an \mbox{\(\tilde{m}\)-approximate}-cut~\((B_i, W_i)\) in~\(G[S_i]\) with \(e_{G[S_i]} (B_i, W_i) \leq t \Delta(G)\). Similar to Case~2b, \(\tilde{V} := (M \setminus S_i) \cup B_i\) satisfies~\(m \leq | \tilde{V} | \leq 2m\), see also \cref{figCutPresRCase3}. Consider the graph~\(\tilde{G}\) obtained from~\(G[\tilde{V}]\) by removing all edges in~\(E_G(i)\) and note that~\(\tilde{G}\) does not contain any edge between the vertices in~\( B_i\) and the vertices in~\(\tilde{V} \setminus B_i \) due to \cref{propCutPlabeling}. The remaining part of Case~3 is analog to Case~2b: First, a tree decomposition~\( (\tilde{T},\tilde{\X})\) of~\(\tilde{G}\) with~\(r(\tilde{T},\tilde{\X}) \geq \frac{1}{2}r\) and width at most~\(t-1\) is constructed. Then, \cref{thmGenTreeDiam} can be used to obtain a set~\(\tilde{B} \subseteq \tilde{V}\) with~\(|\tilde{B}| =m\) such that~\((B,W)\) with~\(B:=\tilde{B}\) and~\(W := V \setminus B\) is a cut in~\(G\) with the desired properties. This completes the proof of \cref{lemmaCutPresR}.

\subsection{Algorithm for Tree-Like Graphs}
\label{subsecGenAlgo}

The goal of this section is to prove the algorithmic part of \cref{thmGenKSec}, i.e., to argue that there is an algorithm that, when given a tree decomposition~\((T,\X)\) of a graph~\(G\) on~\(n\) vertices, computes a \ksec{k} in~\(G\) with width within the bound stated in \cref{thmGenKSec} in~\(\bigO(k \|(T,\X)\| )\) time.  In general, it is \mbox{NP-hard} to compute a tree decomposition of minimum width~\cite{ArnborgCorneil} but, for fixed~\(t\in \mathbb{N}\), there is an algorithm that, when given a graph~\(G\) with~\({\tw(G) \leq t-1}\), computes a tree decomposition of width at most~\(t-1\) in linear time \cite{Bodlaender96}. Here, we assume that a tree decomposition of the graph~\(G\) is provided as input. 

For the implementation we always assume that the input graph~\(G\) satisfies~\(V(G)=[n]\) for some integer~\(n\), and that the clusters of the provided tree decomposition~\((T,\X)\) are given as unordered lists. Moreover, we assume that~\(T\) is given by its adjacency lists and that each node of~\(T\) has a link pointing to its cluster. The algorithm described here only uses~\((T,\X)\) and not the graph~\(G\) itself. Sets, and in particular the sets~\(B_{\ell}\) of the desired \ksec{k}, are stored as unordered lists of vertices of~\(G\). Therefore, the union of two disjoint sets is a simple concatenation of lists and takes constant time. \cref{tableRunningTimes} gives an overview on the subroutines that are used here and discussed briefly in~\cite{PaperGenTreeDiam} and in detail in~\cite{thesisTina}. The following description of the algorithm contained in \cref{thmGenKSec} focuses on its main aspects. A detailed description can be found in Chapter~6.3 of~\cite{thesisTina}.

\begin{table}[b]
\begin{tabular}{l@{\hspace{1em}}c@{\hspace{1em}}c@{\hspace{1em}}l}
  \toprule
  Algorithm/Task & Running Time & Details \\
  \midrule
  approximate cut (\cref{lemmaApproxCutGen})  & \( \bigO( \| (T,\X) \| ) \) & Lemma~4 in~\cite{PaperGenTreeDiam}\\ 
  induced tree decomposition for a subgraph & \( \bigO(\| (T,\X) \|) \) &  clear \\
  make~\((T,\X)\) nonredundant (\cref{propTDNonRed}) & \( \bigO( \|(T,\X)\| ) \) & Proposition~20 in~\cite{PaperGenTreeDiam}\\ 
  heaviest path in~\( (T,\X)\) & \( \bigO( \| (T,\X) \| ) \) & Lemma~21 in~\cite{PaperGenTreeDiam} \\
  \mbox{\(P\)-labeling} for a path~\( P \subseteq T\) (\cref{lemmaPLabeling}) & \( \bigO( \| (T,\X) \| ) \) & Lemma~22 in~\cite{PaperGenTreeDiam} \\
  \bottomrule
\end{tabular}
\caption{Overview on subroutines described in~\cite{PaperGenTreeDiam}. The input for each sub\-rou\-tine is a tree decomposition~\( (T,\X)\) of an arbitrary graph with vertex set~\([n]\) for some integer~\(n\).}
\label{tableRunningTimes}
\end{table}

Consider a tree decomposition~\((T_0 ,\X_0)\) of some graph~\(G_0\) and let~\(G\) be some subgraph of~\(G_0\). When given a list of the vertices in~\(G\), it is easy to traverse~\(T_0\) and delete all vertices not in~\(G\) from the clusters in~\(\X_0\) in order to compute the induced tree decomposition~\((T,\X)\) of~\(G\) with respect to~\((T_0,\X_0)\) in time proportional to~\(\|(T_0,\X_0)\|\). To satisfy the requirement~\(V(G) = [n]\) for some integer~\(n\), which is needed for the subroutines in \cref{tableRunningTimes}, a bijection between~\(V(G)\) and~\([n]\) can be set up while computing~\((T,\X)\). Alternatively, to avoid the relabeling, the same arrays can be used in all calls of the subroutines with a single initialization in the beginning. Therefore, it suffices to argue that a cut with the properties in \cref{lemmaCutPresR} can be computed in~\(\bigO(\|(T,\X)\|)\) time. 

Consider a tree decomposition~\((T,\X)\) of some graph~\(G\) on~\(n\) vertices with~\(V(G)=[n]\) and fix an integer~\(m \in [n]\).  The algorithm described here follows the construction from \cref{subsubsecGenCases} and uses the notation from \cref{subsubsecGenNotation}. Due to \cref{propTDNonRed}, we may assume that~\((T,\X)\) is nonredundant. Computing a heaviest path~\(P \subseteq T\) and a \mbox{\(P\)-labeling} takes~\(\bigO(\|(T,\X)\|)\) time according to \cref{tableRunningTimes}. While doing so, further parameters related to the labeling can be computed as stated by the next lemma, which is also from~\cite{PaperGenTreeDiam}. 

\begin{lemma}[Lemma~22 in~\cite{PaperGenTreeDiam}]\label{lemmaPLabeling}
  Given a tree decomposition~\((T,\X)\) of a graph~\(G=(V,E)\) on~\(n\) vertices with~\(V=[n]\) and a path~\(P \subseteq T\), a \mbox{\(P\)-labeling} of~\(G\)  can be computed in~\(\bigO(\|(T,\X)\|)\) time. While doing so, the following parameters can be computed (using the notation from Section~\ref{subsubsecGenNotation}):
  \begin{itemize}[leftmargin= \IdentationTDConditions]
    \item two integer arrays~\(A_L\) and~\(A_V\), each of length~\(n\), such that for~\(x \in V\) the entry~\(A_L[x]\) is the label of vertex~\(x\) and for~\(\ell \in [n]\) the entry~\(A_V[\ell]\) is the vertex that received label~\(\ell\), 
    \item a binary array~\(A_R\) of length~\(n\), such that for~\(x \in V\) the entry~\(A_R[x]\) is one if and only if~\(x \in R\), 
    \item an integer array~\(A_P\) of length~\(n\), such that for~\(x\in V\) the entry~\(A_P[x]\) is the path node of~\(x\), and
    \item a list~\(L_P\) of the nodes on the path~\(P\) in the order in which they occur when traversing~\(P\), including, for each~\(h \in V_P\), a pointer to the root of~\(T_h\) stored as an arborescence with root~\(h\).
  \end{itemize}
\end{lemma}

From now on,~\(A_L\), \(A_V\), \(A_R\), \(A_P\), and~\(L_P\) denote the arrays and the list from the previous lemma. As in \cref{subsecTreeAlgo}, in the implementation, vertices and labels are not identified. The arrays~\(A_L\) and~\(A_V\) allow to convert vertices to labels and vice versa in constant time. For~\(x \in V\) denote by~\(d_1(x)\) the \mbox{\(R\)-distance} of the vertex with label~\(1\) and~\(x\). Observe that by using~\(A_V\) and~\(A_R\), all values~\(d_1(x)\) can be computed simultaneously in~\(\bigO(n)\) time. Then, \(d_R(x,y) = d_1(y) -d_1(x)\) for all~\(x,y \in V\) where~\(x\) is smaller than~\(y\) and, thus, a vertex~\(v \in V\) with~\(d_R(v, v+m) = \Bfloor{rm}\) and~\(v \in R\) or~\(v+m \in R\) can be found in~\(\bigO(n)\) time. The set~\(M := \{v, v+1, \dots, v+m-1\}\) can be read off the array~\(A_V\) in~\(\bigO(n)\) time. Using the array~\(A_R\), the algorithm can determine in constant time which of the cases from \cref{subsubsecGenCases} applies. If Case~1 or Case~2a applies, there is nothing more to do. So, assume that Case~2b applies, i.e.,~\(v\in R\) as well as~\(v+m-1 \not\in R\) and~\(v+m \not\in R\). If Case~3 applies, the algorithm can be implemented similarly to Case~2b. With the array~\(A_P\) the path-node~\(j\) of~\(v+m\) can be determined in constant time. For each~\(h \in V_P\) and each~\(x \in V\) the following holds
\[ x \in S_h \quad \quad \Leftrightarrow \quad \quad \text{\(x \not\in R\) and the path-node of~\(x\) is~\(h\)} \quad \quad \Leftrightarrow \quad \quad A_R[x] = 0 \text{ and } A_P[x]=h. \]
Hence, for each~\(x \in V\), it takes constant time to check whether~\(x\) lies in~\(S_j\) and the algorithm can compute a list of the vertices in~\(M \setminus S_j\) in~\(\bigO(n)\) time. Furthermore, the induced tree decomposition for~\(G[S_j]\), say~\((\hat{T}, \hat{\X})\), can be computed in~\(\bigO(\|(T,\X)\|)\) time. Keeping track of the vertex with the smallest label in~\(S_j\), it is easy to shift the labels such that the requirement~\(V(\hat{G}) = [\hat{n}]\) for some integer~\(\hat{n}\) is satisfied for the underlying graph~\(\hat{G} \approx G[S_j]\). Now, \cref{lemmaApproxCutGen} implies that the \mbox{\(\tilde{m}\)-approximate} cut~\((B_j, W_j)\) in~\(G[S_j]\) can be computed in time proportional to~\(\| (\hat{T}, \hat{\X}) \| \leq \|(T, \X)\|\). Recall the construction of the tree decomposition~\((\tilde{T}, \tilde{\X})\) for the graph~\(\tilde{G}\). Computing the tree decompositions~\((\tilde{T}_1,\tilde{\X}_1)\) and~\((\tilde{T}_2,\tilde{\X}_2)\) as well as the paths~\(\tilde{P}_1\) and~\(\tilde{P}_2\) takes~\(\bigO(\|(T,\X)\|)\) time. Hence, also~\((\tilde{T}, \tilde{\X})\) can be computed in time proportional to~\(\|(T,\X)\|\) and satisfies~\(\|(\tilde{T}, \tilde{\X})\| \leq 2 \|(T,\X)\|\). So, applying the algorithm contained in \cref{thmGenTreeDiam} to~\(\tilde{G}\) with the tree decomposition~\((\tilde{T}, \tilde{\X})\) requires~\(\bigO(\|(T,\X)\|)\) time and yields the set~\(\tilde{B}\). Using that~\(n \leq \|(T,\X)\|\) due to~(T1), the desired set~\(B=\tilde{B}\) is computed in time proportional to~\(\|(T,\X)\|\).

\bibliographystyle{abbrv}
\addcontentsline{toc}{section}{References} 
\bibliography{references}
\addcontentsline{toc}{section}{References} 

\end{document}